\documentclass[12pt,leqno,a4paper]{amsart}
\usepackage{amssymb}
\usepackage{enumitem}
\usepackage[mathscr]{euscript}

\newcommand{\FF}{{\mathbb{F}}}
\newcommand{\ZZ}{\mathbb{Z}}

\newcommand{\bL}{{\mathbf{L}}}
\newcommand{\bG}{{\mathbf{G}}}
\newcommand{\bH}{{\mathbf{H}}}

\newcommand{\cF}{{\mathcal{F}}}
\newcommand{\cO}{{\mathcal{O}}}

\newcommand{\fS}{{\mathfrak{S}}}

\newcommand{\scrL}{{\mathscr{L}}}

\newcommand{\AV}{{\operatorname{AV}}}
\newcommand{\Hom}{{\operatorname{Hom}}}
\newcommand{\Irr}{{\operatorname{Irr}}}
\newcommand{\Ind}{{\operatorname{Ind}}}
\newcommand{\Res}{{\operatorname{Res}}}
\newcommand{\GL}{{\operatorname{GL}}}
\newcommand{\PGL}{{\operatorname{PGL}}}
\newcommand{\SL}{{\operatorname{SL}}}
\newcommand{\GU}{{\operatorname{GU}}}
\newcommand{\ad}{{\operatorname{ad}}}
\newcommand{\uni}{{\operatorname{uni}}}
\newcommand\RLG{{R_L^G}}
\newcommand\RLiG{{R_{L_i}^G}}
\newcommand\RtLtG{{R_{\widetilde{L}}^{\widetilde{G}}}}

\newcommand{\tw}[1]{{}^{#1}\!}

\let\wt=\widetilde

\newtheorem{thm}{Theorem}[section]

\newtheorem{prop}[thm]{Proposition}
\newtheorem{cor}[thm]{Corollary}

\newtheorem*{thmA}{Theorem 1}

\theoremstyle{remark}
\newtheorem{rem}[thm]{Remark}

\begin{document}

\title[Modular irreducibility of cuspidal unipotent characters]{Modular irreducibility\\
of cuspidal unipotent characters}

\date{\today}

\author{Olivier Dudas}
\address{Universit\'e Paris Diderot, UFR de Math\'ematiques,
B\^atiment Sophie Germain, 5 rue Thomas Mann, 75205 Paris CEDEX 13, France.}
\email{olivier.dudas@imj-prg.fr}

\author{Gunter Malle}
\address{FB Mathematik, TU Kaiserslautern, Postfach 3049,
         67653 Kaisers\-lautern, Germany.}
\email{malle@mathematik.uni-kl.de}

\thanks{The first author gratefully acknowledges financial support by
the ANR grant ANR-16-CE40-0010-01.
The second author gratefully acknowledges financial support by ERC
  Advanced Grant 291512.}



\begin{abstract}
We prove a long-standing conjecture of Geck which predicts that cuspidal
unipotent characters remain irreducible after $\ell$-reduction. To this end,
we construct a progenerator for the category of representations of a finite
reductive group coming from generalised Gelfand--Graev representations. This
is achieved by showing that cuspidal representations appear in the head of
generalised Gelfand--Graev representations attached to cuspidal unipotent
classes, as defined and studied in \cite{GM96}.
\end{abstract}

\maketitle


\section{Introduction} \label{sec:intro}
Let $\bG$ be a connected reductive linear algebraic group defined over a
finite field of characteristic~$p>0$ with corresponding Frobenius
endomorphism $F$. This paper is devoted to
the proof of a long-standing conjecture of Geck regarding cuspidal unipotent
characters of the finite reductive group $\bG^F$ (see \cite[(6.6)]{Ge92}):

\begin{thmA}   \label{thm:A}
 Assume that $p$ and $\ell$ are good for $\bG$ and that $\ell \nmid
 p|Z(\bG)^F/Z^\circ(\bG)^F|$. Then any cuspidal unipotent character of $\bG^F$
 remains irreducible under reduction modulo~$\ell$.
\end{thmA}

This property was shown to be of the utmost importance in order to relate
Harish-Chandra series in characteristic zero and $\ell$, see
\cite[Prop.~6.5]{Ge92}, and for the study of supercuspidal representations,
see Hiss \cite{Hi96}. The conjecture had previously only been shown in special
cases: for $\GU_n(q)$ by Geck \cite{Ge91}, for classical groups at linear
primes $\ell$ by Gruber--Hiss \cite{GrHi} and at the prime~$\ell=2$ by Geck
and the second author \cite{GM96}, as well as for some exceptional groups for
which the $\ell$-modular decomposition matrix is known. 

\smallskip

Our strategy for the proof of this conjecture is the construction of a
progenerator of the category of representations of $\bG^F$ over a field of
positive characteristic $\ell\ne p$ (non-defining characteristic). We produce
such a progenerator $P$ using generalised Gelfand--Graev representations
associated to suitably chosen unipotent classes.
We then show that cuspidal unipotent characters occur with multiplicity
one in $P$, which is enough to deduce the truth of Geck's conjecture.
\smallskip

Generalised Gelfand--Graev representations (GGGRs) are a family of projective
representations $\{\Gamma_C^G\}$ labelled by unipotent classes $C$ of $G=\bG^F$.
They are defined whenever the characteristic~$p$ is good for $\bG$.
The sum of all the GGGRs is a progenerator for the representations of $G$,
since the GGGR corresponding to the trivial class is the regular representation,
hence a progenerator itself. Our construction relies on the work of Geck and
the second author \cite{GM96} who showed that any representation $\Gamma_C^G$
can be replaced by the Harish-Chandra induction of a GGGR from a suitable Levi
subgroup $L$ of $G$, up to adding and removing GGGRs corresponding to unipotent
classes larger than $C$ for the closure ordering. This led them to the notion
of \emph{cuspidal classes} for which no such proper Levi subgroup exists.
Following their observation, we consider the projective module
$$ P = \bigoplus_{(L,C)} R_L^G(\Gamma_C^L)$$
where $L$ runs over the set of $1$-split Levi subgroups of $G$ (\emph{i.e.},
Levi complements of rational parabolic subgroups) and $C$ over the set of
cuspidal unipotent classes of $L$.

We show that $P$ is a progenerator (see Corollary \ref{cor:progenerator}).
To this end we adapt the work of Geck--H\'ezard \cite{GeHe08} on a
conjecture of Kawanaka to prove that the family of characters of Harish-Chandra
induced GGGRs forms a basis of the space of unipotently supported class
functions. As a consequence, we obtain that cuspidal modules must appear in
the head of GGGRs associated to cuspidal classes. We believe that this should
be of considerable interest for studying projective covers of cuspidal modules,
as there exist very few cuspidal classes in general. For example, when $G$ is
a group of type $A$, cuspidal classes are regular classes and only usual
Gelfand--Graev representations are needed to define $P$. In this specific
case, such a progenerator already appears for example in work of Bonnaf\'e
and Rouquier \cite{BR06,Bo11}.
\smallskip

This paper is organised as follows. Section~\ref{sec:progen} is devoted to the
construction of the progenerator. In Corollary~\ref{cor:progenerator} we show
how to construct it from generalised Gelfand--Graev representations.
Section~\ref{sec:irrcusp} contains our main application on the $\ell$-reduction
of cuspidal unipotent characters, with the proof of Theorem~1.
\medskip

\noindent
{\bf Acknowledgement:} We thank Meinolf Geck and Jay Taylor for valuable
comments on an earlier version.

\section{A progenerator}   \label{sec:progen}

Let $\bG$ be a connected reductive linear algebraic group over
$\overline{\FF_p}$ and $F$ be a Frobenius endomorphism endowing $\bG$ with
an $\FF_q$-structure. If $\bH$ is any $F$-stable closed subgroup of $\bG$,
we denote by $H := \bH^F$ the finite group of $\FF_q$-points in $\bH$.

We let $\ell\ne p$ be a prime and let $(K,\cO,k)$ denote a splitting
$\ell$-modular system for $G$. We will consider representations of $G$
over one of the rings $K$, $\cO$, or $k$.

Throughout this section, we will always assume that $p$ is good for $\bG$.
The results on generalised Gelfand--Graev representations that we shall need
were originally proved by Kawanaka \cite{Kaw82} and Lusztig \cite{Lu92},
under some restriction on $p$ and $q$.
This restriction was recently removed by Taylor \cite{Tay14}, so that we can
work under the assumption that $p$ is good for $\bG$. Note that this is already
required for the classification of unipotent classes to be independent from $p$.

\subsection{Unipotent support of unipotent characters}   \label{sec:duality}
Given $\rho\in \Irr(G)$ and $C$ an $F$-stable unipotent class of $\bG$, we
denote by $\AV(C,\rho) = |C^F|^{-1}\sum_{g \in C^F}\rho(g)$ the average value
of $\rho$ on $C^F$. We say that $C$ is a \emph{unipotent support} of $\rho$
if $C$ has maximal dimension for the property that $\AV(C,\rho)\neq 0$. 
Geck \cite[Thm.~1.4]{Ge96} has shown that whenever $p$ is good for $\bG$,
any irreducible character $\rho$ of $G$ has a unique unipotent support, which
we will denote by $C_\rho$.

By \cite[\S11]{Lu92} (see \cite[\S14]{Tay14} for the extension to any
good characteristic), unipotent supports of unipotent characters are special
classes. They can be computed as follows: any family $\cF$ of the Weyl group
of $\bG$ contains a unique special representation, which is the image under
the Springer correspondence of the trivial local system on a special
unipotent class $C_\cF$. Then this class is the common unipotent support
of all the unipotent characters in $\cF$.

\subsection{Generalised Gelfand--Graev representations}
Given an $F$-stable unipotent element $u \in G$, we denote by $\Gamma_u^G$, or
simply $\Gamma_u$, the \emph{generalised Gelfand--Graev representation}
associated with $u$. It is an $\cO G$-lattice. The construction is given for
example in \cite[\S3.1.2]{Kaw82} (with some extra assumption on $p$) or in
\cite[\S5]{Tay14}. The first elementary properties that can be deduced are
\begin{itemize}
 \item if $\ell \neq p$, then $\Gamma_u$ is a projective $\cO G$-module;
 \item if $u$ and $u'$ are conjugate under $G$ then $\Gamma_u\cong\Gamma_{u'}$.
\end{itemize}
The character of $K\Gamma_u$ is the \emph{generalised Gelfand--Graev character}
associated with $u$. We denote it by $\gamma_u^G$, or simply $\gamma_u$. It
depends only on the
$G$-conjugacy class of $u$. When $u$ is a regular unipotent element then
$\gamma_u$ is a usual Gelfand--Graev character as in \cite[\S14]{DM91}.
\smallskip

Lusztig \cite[Thm.~11.2]{Lu92} (see Taylor \cite{Tay14} for
the extension to good characteristic) gave a condition on the unipotent
support of a character to occur in a generalised Gelfand-Graev character.
Namely, given $\rho \in \Irr(G)$ and $\rho^* \in \Irr(G)$ its Alvis-Curtis dual
(see \cite[\S8]{DM91})
 \begin{itemize}
  \item there exists $u\in C_{\rho^*}^F$ such that
   $\langle\gamma_u;\rho\rangle \neq 0$;
  \item if $C$ is an $F$-stable unipotent conjugacy class of $\bG$ such that
   $\dim C > \dim C_{\rho^*}$ then $\langle \gamma_{u};\rho\rangle = 0$ for
   all $u \in {C}^F$.
\end{itemize}

\subsection{Cuspidal unipotent classes} Following Geck and Malle \cite{GM96} we
say that an $F$-stable unipotent class $C$ of $\bG$ is \emph{non-cuspidal} if
there exists a 1-split \emph{proper} Levi subgroup $\bL$ of $\bG$ such that
\begin{itemize}
 \item $C \cap L \neq \emptyset$;
 \item for all $u \in C\cap L$, the natural map
  $C_\bL(u)/C_\bL^\circ(u) \longrightarrow C_\bG(u)/C_\bG^\circ(u)$ is an
  isomorphism.
\end{itemize}
Here recall that a \emph{1-split Levi subgroup} of $(\bG,F)$ is by definition
an $F$-stable Levi complement of an $F$-stable parabolic subgroup of $\bG$.
If no such proper Levi subgroup exists, we say that the class $C$ is
\emph{cuspidal}. Note that cuspidality is preserved under the quotient map
$\bG \rightarrow \bG/Z^\circ(\bG)$. In particular, when $\bG$ has connected
centre, a unipotent class $C$ is cuspidal if and only if its image in the
adjoint quotient $\bG_\ad$ (with same root system as $\bG$) is cuspidal.

\subsection{A progenerator}   \label{subsec:progen}
Recall that $\bG$ is connected reductive in
characteristic~$p$ and that $\ell \neq p$. In particular every generalised
Gelfand--Graev representation of $\cO G$ is projective.
\smallskip

When $\bG=\GL_n$, it is known from \cite[Thm.~7.8]{GHM} that any cuspidal
$kG$-module $N$ lifts to characteristic zero in a (necessarily) cuspidal
$KG$-module. The latter is a constituent of some Gelfand--Graev
representation $\Gamma$ of $G$ so that $N$ is in the head of $k\Gamma$.
We prove an analogue of this result for $\bG$ of arbitrary type.

\begin{thm}   \label{thm:cusp}
 Assume that $p$ is good for $\bG$ and that $\ell \neq p$. Let $N$ be a cuspidal
 $kG$-module. Then there exists an $F$-stable unipotent class $C$ of $\bG$ which
 is cuspidal for $\bG_\ad$ and $u \in C^F$ such that
 $\Hom_{kG}(k\Gamma_u, N) \neq 0$.
\end{thm}

This results from the following version of a conjecture by Kawanaka
that was proved by Geck and H\'ezard \cite[Thm.~4.5]{GeHe08}.

\begin{prop}   \label{prop:unipfunctions}
 Assume that the centre of $\bG$ is connected. Then the $\ZZ$-module of
 unipotently supported virtual characters of $G$ is generated by
 $\{\RLG(\gamma_u^L)\}$ where $\bL$ runs over 1-split Levi subgroups of $\bG$,
 and $u$ over $F$-stable unipotent elements of cuspidal unipotent classes
 of~$\bL$.
\end{prop}

\begin{proof}
Let $C_1, \ldots, C_N$ be the $F$-stable unipotent classes of $\bG$, ordered
by increasing dimension. For each $C_i$, we choose a system of representatives
$u_{i,1},u_{i,2},\ldots$ of the $G$-orbits in $C_i^F$. In \cite[\S 4]{GeHe08},
it is shown, under the assumption that $p$ is large, that to each $u_{i,r}$
one can associate an irreducible character $\rho_{i,r}$ of $G$ such that
\begin{itemize}
 \item[(1)] $\rho_{i,r}$ has unipotent support $C_i$;
 \item[(2)] $\langle D_{G}(\rho_{i,r});\gamma_{u_{i,s}}\rangle
   = \pm \delta_{r,s}$ for all $s$.
\end{itemize}
Thus the $\{D_G(\rho_{i,r})\}$ span the $\ZZ$-module of unipotently supported
virtual characters of $G$. Here $D_G$ denotes the Alvis--Curtis
duality on characters (so $D_G(\rho)=\pm\rho^*$ in our earlier notation).
Let us explain why the arguments in \cite[\S 4]{GeHe08} can be generalised
to the case of good characteristic.
For a given $i$, the irreducible characters $\{\rho_{i,1},\rho_{i,2},\ldots\}$
are obtained as characters lying in a family of a Lusztig series belonging to
an isolated element, and whose associated unipotent support is $C_i$. When the
finite group attached by Lusztig to the family is abelian (resp.~isomorphic to
$\fS_3$), these characters are constructed using \cite[Prop.~6.6]{Ge97}
(resp.~\cite[Prop.~6.7]{Ge97}). As mentioned in \cite[\S 2.4]{Ge97}, this
requires a generalisation of some of Lusztig's results on generalised
Gelfand--Graev representations \cite{Lu92} to the case of good characteristic.
This was recently achieved by Taylor \cite{Tay14}.
Finally, when the finite group attached to the family is isomorphic to $\fS_4$
or $\fS_5$, then $C$ is a specific special unipotent class of $F_4$ or $E_8$.
In that case one can use the results in \cite[\S 4]{DLM14} which hold whenever
$p$ is good or $\bG$, again thanks to \cite{Tay14}.
\smallskip

Now let us choose the system of representatives $u_{i,s}\in C_i$ in a minimal
Levi subgroup $L_i$ given by \cite[Thm.~3.2]{GM96}. Then property~(2) is still
satisfied if we replace $\gamma_{u_{i,s}}^G$ by $\RLiG(\gamma_{u_{i,s}}^{L_i})$,
see \cite[Cor.~2.7]{GM96} and \cite[Thm.~8.11]{DM91}. Moreover, for $j<i$ we
have $\langle D_{G}(\rho_{j,r}); \RLiG(\gamma_{u_{i,s}}^{L_i}) \rangle = 0$
since $\rho_{j,r}$ vanishes on the support of
$D_G\big(\RLiG(\gamma_{u_{i,s}}^{L_i})\big)$.
Indeed, by \cite[Prop.~2.3]{GM96}, this support is contained in the union of
unipotent classes $C_l$ satisfying $C_i \subset \overline{C}_l$. But if
$C_i \subset \overline{C}_j$ with $\dim C_j \leq \dim C_i$ we would have
$C_i = C_j$. Therefore the matrix
$$\Big( \left\langle D_{G}(\rho_{j,r}); \RLiG( \gamma_{u_{i,s}}^{L_i})
  \right\rangle_{G} \Big)_{j,r;i,s}$$
is block upper diagonal with identity blocks on the diagonal, hence invertible.
Therefore, $\{\RLiG( \gamma_{u_{i,s}}^{L_i})\}$ also spans the
$\ZZ$-module of unipotently supported virtual characters of~$G$.
\end{proof}

\begin{proof}[Proof of Theorem~\ref{thm:cusp}]
Let $\bG \hookrightarrow \wt{\bG}$ be a regular embedding, compatible with $F$,
that is $\wt{\bG}$ has connected centre and same derived subgroup as $\bG$.
Note that this restricts to an isomorphism on the variety of unipotent elements.
To avoid any confusion, given a unipotent element $u$ in $\wt G$ we shall
denote by $\wt \Gamma_u$ (resp.~$\Gamma_u$) the corresponding generalised
Gelfand--Graev representation of $\wt G$ (resp.~$G$) and by $\wt \gamma_u$
(resp.~$\gamma_u$) its character. By construction we have
$\wt \Gamma_u=\Ind^{\wt G}_G\,\Gamma_u$. Fix a system of coset representatives
$g_1,\ldots,g_r$ of $\wt G/G$. Since $G\unlhd\wt G$, the Mackey formula yields
an isomorphism of functors
\begin{equation}\label{eq:mackey}
  \Res^{\wt G}_G \circ \Ind^{\wt G}_G \simeq
   \bigoplus_{i=1}^r \ad(g_i).
\end{equation}
In particular $\Res_G^{\wt G}\, \wt \Gamma_u \simeq
\bigoplus \Gamma_{g_i u g_i^{-1}}$ (using that $\ad(g_i) \big(\Gamma_u\big) = 
\Gamma_{g_i u g_i^{-1}}$, see \cite[Prop.~2.2]{Ge93}).

\smallskip
Let $N$ be a cuspidal simple $kG$-module. From \eqref{eq:mackey} we deduce
that $\Ind^{\wt G}_G\, N$ is a cuspidal $k\wt G$-module. Let $M$ be a simple
constituent of the socle of $\Ind^{\wt G}_G \, N$ and let $\psi$ be its Brauer
character. We denote by $\psi_\uni$ the unipotently supported class function
which coincides with $\psi$ on the set of unipotent elements. Then for every
unipotent element $u$ contained in a Levi subgroup $\wt L$ of $\wt G$ we have
$$\langle \RtLtG(\wt\gamma_u^{\tilde L});\psi_\uni \rangle_{\wt G}
  = \langle \RtLtG(\wt\gamma_u^{\tilde L});\psi \rangle_{\wt G},$$
which is zero whenever $\wt L \neq \wt G$ since $\psi$ is
cuspidal. Together with Proposition \ref{prop:unipfunctions}, we deduce that
there exists an $F$-stable cuspidal unipotent class $C$ of $\wt \bG$ and
$u \in C^F$ such that $\langle \wt\gamma_u;\psi \rangle_{\wt G} \neq 0$. By
construction of the generalised Gelfand--Graev characters there exist a
unipotent subgroup $U_{1.5}$ of $\wt G$ and a linear character $\chi_u$ of
$U_{1.5}$ such that $\wt\gamma_u = \Ind_{U_{1.5}}^{\wt G}\, \chi_u$.
If $k_{\chi_u}$ denotes the $1$-dimensional $kU_{1.5}$-module on which
$U_{1.5}$ acts by $\chi_u$ then since $U_{1.5}$ is a $p$-group, hence an
$\ell'$-group, we get
$$\begin{aligned}
  \dim \Hom_{k\wt G}(k\wt\Gamma_u, M) & \,= \dim \Hom_{k{U_{1.5}}}
  (k_{\chi_u}, \Res_{{U_{1,5}}}^{\wt G} M) \\
  & \,=  \langle \chi_u; \Res_{{U_{1.5}}}^{\wt G} \psi\rangle_{U_{1.5}} \\
  & \, =  \langle \wt\gamma_u ; \psi\rangle_{\wt G}
\end{aligned}$$
which is non-zero. This proves that there is a surjective map
$k\wt\Gamma_u \twoheadrightarrow M$. By restriction to $G$ we get surjective
maps $\Res_G^{\wt G}\, k\wt\Gamma_u \twoheadrightarrow \Res_G^{\wt G}\, M
\twoheadrightarrow N$. It follows that
$\Hom_{kG}(k\Gamma_{g_i u g_i^{-1}}, N) \neq 0$ for at least one
$i\in\{1,\ldots,r\}$, thus proving the claim.
\end{proof}

Given $C$ an $F$-stable cuspidal unipotent class of $\bG$ we will denote by
$\Gamma_C$ the sum of all the generalised Gelfand--Graev representations
corresponding to representatives of $G$-orbits in $C^F$. If $N$ is a simple
$kG$-module then there exist a 1-split Levi subgroup $\bL$ of $\bG$ and a
cuspidal $kL$-module $M$ such that $\Hom_{kG}(\RLG(M),N)\ne0$. Consequently,
we can build a progenerator of $kG$ using Theorem~\ref{thm:cusp} for all
1-split Levi subgroups. For this, define $\scrL$ to be the set of pairs
$(L,C)$ such that $\bL$ is a 1-split Levi subgroup of $\bG$, and $C$ is an
$F$-stable unipotent class of $\bL$ which is cuspidal for $\bL_\ad$

\begin{cor}   \label{cor:progenerator}
 Let $\bG$ be connected reductive with Frobenius map $F$ and assume that $p$ is
 good for $\bG$. Then the module
 $$ P = \bigoplus_{(L,C) \in \scrL} \RLG(k\Gamma_C^L)$$
 is a progenerator of $kG$.
\end{cor}

\begin{rem}
Except for groups of type $A$, even the multiplicities of unipotent characters
in the character of $P$ depend on $q$ (e.g. the multiplicity of the Steinberg
character).
\end{rem}

\begin{rem}
One can use Theorem \ref{thm:cusp} to reprove that there is at most one
unipotent cuspidal module in $\GL_n(q)$ (see for example
\cite[Thm.~5.21 and Cor.~5.23]{Di98}). Indeed, the only
cuspidal class in $\PGL_n(\overline{\FF_p})$ is the regular class, therefore
any cuspidal module must appear in the head of the usual Gelfand--Graev
representation $\Gamma_u$ for some (any) regular unipotent element $u$.
Since $\gamma_u$ has only the Steinberg character as a unipotent constituent,
it follows that $\Gamma_u$ has only one unipotent projective indecomposable
summand, therefore at most one unipotent cuspidal module in its head.

For general linear groups again, the progenerator given in Corollary
\ref{cor:progenerator} involves only parabolic induction of usual Gelfand--Graev
representations. This progenerator was already studied in \cite{BR06} and
\cite{Bo11}. Our construction is a natural generalisation to arbitrary finite
reductive groups.
\end{rem}

\section{$\ell$-reduction of cuspidal unipotent characters}   \label{sec:irrcusp}

Recall that $\bG$ is a connected reductive group defined over $\FF_q$, with
corresponding Frobenius endomorphism $F$. Throughout this section we will
assume that $p$, the characteristic of $\FF_q$, is good for $\bG$. In that
case we can use the result of the previous section to show that under some
mild assumptions on $\ell$, cuspidal unipotent characters remain irreducible
after $\ell$-reduction.

\subsection{Multiplicities in generalised Gelfand--Graev characters}
Recall from \S\ref{sec:progen} that given a unipotent character $\rho$,
any generalised Gelfand--Graev character $\gamma_u$ with
$u \notin C_{\rho^*}$ and $\dim (u) \geq \dim C_{\rho^*}$ satisfies
$\langle \gamma_u ; \rho \rangle = 0$. We combine here results from
\cite{GM96,GeHe08} to compute $\langle \gamma_u ; \rho \rangle$ when $\rho$
is cuspidal and $u \in C_{\rho^*}$.

Given an $F$-stable unipotent class $C$ of $\bG$, recall that $\Gamma_C$ is the
sum of all the $\Gamma_u$'s where $u$ runs over a set of representatives of
$G$-orbits in $C^F$. We will denote by $\gamma_C$ the character of
$K\Gamma_C$.

\begin{prop}   \label{prop:mult}
 Assume that $\bG$ is simple of adjoint type. Let $\rho$ be a cuspidal
 unipotent character with unipotent support $C_\rho$. Then
 $C_\rho = C_{\rho^*}$ and $\langle \gamma_{C_\rho} ; \rho \rangle = 1$.
\end{prop}

\begin{proof}
Let $\cF$ be the family of unipotent characters containing $\rho$. Since
$\rho$ is cuspidal, we have $\rho = \rho^*$, and hence $C_{\rho^*} = C_\rho$.
By \cite[Thm.~3.3]{GM96}
$C_\rho$ is a cuspidal class, and for $u \in C_\rho^F$ the finite group
$A_G(u)$ is isomorphic to the small finite group associated to $\cF$ as in
\cite[\S4]{LuB}. In particular, the condition ($*$) in \cite[Prop.~2.3]{GeHe08}
holds for $C_\rho$ (and for $C_{\rho^*}$).
\smallskip

If $\bG$ is of classical type or of type $E_7$, then the claim follows
from \cite[Prop.~4.3]{GeHe08} since $A_G(u)$ is abelian in that case.
\smallskip

If $\bG$ is of exceptional type different from $E_7$, then $A_G(u)$ equals
$\fS_3$ (for types $G_2$ and $E_6$), $\fS_4$ (for type $F_4$) or $\fS_5$
(for type $E_8$). In that case, by \cite[Lemma~1.3.1]{Kaw86}, there is at
most one local system on $C_\rho$ that is not in the image of the Springer
correspondence and the multiplicity of $\rho$ in $\gamma_u$ was
first computed by Kawanaka \cite{Kaw86}. An explicit formula can also be
found in \cite[\S 6]{DLM14}. The assumption~(6.1) in loc.~cit.~is satisfied
and the projection of $\gamma_u$ to $\cF$ can be explicitly computed from
Lusztig's parametrisation. Namely, let now $u \in C_\rho^F$ be such that $F$
acts trivially on $A = A_G(u)$. Then the conjugacy classes of $A$ are in
bijection with the $G$-orbits in $C_\rho^F$. Given $a \in A$, we fix
a representative $u(a)$ of the corresponding orbit.
Then the unipotent characters in $\cF$ are parametrized by $A$-conjugacy
classes of pairs $(a,\chi)$ where $\chi \in \Irr(C_A(a))$ and according to
\cite[Thm.~6.5(ii)]{DLM14} the projection of $\gamma_{u(a)}$ to $\cF$ is given
by
$$\sum_{\chi \in \Irr(C_A(a))}\chi(1)\,\rho_{(a,\chi)}^*.$$
Now the claim follows from the fact that a cuspidal character always
corresponds to a pair $(a,\chi)$ with $\chi(1) =1$. This can be checked
case-by-case using for example \cite[Appendix]{LuB} (thanks to
\cite[Thm.~1.15, \S1.16]{Lu80}, the parametrisation for $\tw2E_6$ can be
deduced from the one for $E_6$ by Ennola duality).
\end{proof}

\subsection{$\ell$-reduction}
We can now prove the main application of the construction of our progenerator,
viz.~Theorem~1, which we restate in a slightly more general form.

\begin{thm}   \label{thm:irrcusp}
 Assume that $p$ is good for $\bG$ and that $\ell\ne p$. If either $\bG$ is
 simple of adjoint type, or $\ell$ is good for $\bG$ and does not divide
 $|Z(\bG)^F/Z^\circ(\bG)^F|$, then any cuspidal unipotent character of $\bG^F$
 remains irreducible under reduction modulo~$\ell$.
\end{thm}

\begin{proof}
First assume that $\bG$ is simple of adjoint type.
Let $\rho$ be a cuspidal unipotent character of $G$, and $C_\rho$ be its
unipotent support. It is a special and self-dual cuspidal unipotent class.
Since $\rho$ is cuspidal, it does not occur in any $R_L^G(\gamma_u^L)$ for
proper $1$-split Levi subgroups $\bL$ of $\bG$. In addition, if $C$ is a
cuspidal class different from $C_\rho$ then by \cite[Thm.~3.3]{GM96} we have
$C_{\rho^*} = C_\rho \subset \overline{C}$ which forces
$\langle \gamma_u ; \rho \rangle = 0$ for every $u \in C^F$. Finally, it
follows from Proposition~\ref{prop:mult} that
$\langle\gamma_{C_\rho};\rho\rangle=1$. Consequently
$$ \big\langle\sum_{(L,C)\in\scrL} \RLG(\gamma_C^L) ; \rho \big\rangle = 1$$
which by Corollary~\ref{cor:progenerator} proves that $\rho$ appears in the
character of exactly one projective indecomposable module. In other words,
the $\ell$-reduction of $\rho$ is irreducible.
\par
In the general case, let $\bG \hookrightarrow \widetilde\bG$ be a regular
embedding. Then the restrictions of unipotent characters of $\widetilde G$
to $G$ remain irreducible (see \cite[Prop.~17.4]{CE}). Furthermore, under the
assumptions on $\ell$, the unipotent characters form a basic set of the union
of unipotent blocks \cite{Ge93}. Therefore the unipotent parts of the
decomposition matrices of $\widetilde G$ and $G$ are equal, so that we can
assume without loss of generality that the centre of $\bG$ is connected.
Now, $Z(\bG)^F= Z(G)$ and unipotent characters are trivial on the centre.
Consequently the unipotent parts of the decomposition matrices of $G$ and of
$G/Z(\bG)^F$ are equal,
and we can assume that $\bG$ is semisimple of adjoint type. In that case
$\bG$ is a direct product $\bG = \bG_1^{n_1}\times\cdots\times\bG_r^{n_r}$
where each $\bG_i$ is simple of adjoint type, with $F$ cyclically
permuting the copies of $\bG_i$ in $\bG_i^{n_i}$. Now the projection onto the
first component $\bG_i^{n_i} \rightarrow \bG_i$ induces a group isomorphism
$(\bG_i^{n_i})^F \simeq (\bG_i)^{F^{n_i}}$ mapping unipotent characters to
unipotent characters. Therefore we are reduced to the case that $\bG$ is
simple of adjoint type, which was treated above.
\end{proof}

By a result of Hiss \cite[Prop.~3.3]{Hi96}, this has the following consequence
which might be of interest for applications to representations of $p$-adic
groups:

\begin{cor}
 Assume that $p$ and $\ell$ are good for $\bG$ and that $\ell \nmid
 p|Z(\bG)^F/Z^\circ(\bG)^F|$. Then any unipotent supercuspidal simple
 $kG$-module is liftable to an $\cO G$-lattice.
\end{cor}


\end{document}